\documentclass{amsart}
\usepackage{amssymb, amsmath, color}
\usepackage[latin1]{inputenc}
\usepackage{color}

\newtheorem{theorem}{Theorem}
\newtheorem{lemma}[theorem]{Lemma}
\newtheorem{remark}[theorem]{Remark}

\begin{document}

\title[Neutral signature self-dual gradient Ricci solitons]
{Four-dimensional neutral signature self-dual gradient Ricci solitons}
\author{M. Brozos-V\'{a}zquez \, E. Garc\'{i}a-R\'{i}o}
\address{MBV: Departmento de Matem\'{a}ticas, Escola Polit\'ecnica Superior, Universidade da Coru\~na, Spain}
\email{mbrozos@udc.es}
\address{EGR: Faculty of Mathematics,
University of Santiago de Compostela,
15782 Santiago de Compostela, Spain}
\email{eduardo.garcia.rio@usc.es }
\thanks{Supported by projects EM2014/009, GRC2013-045 and MTM2013-41335-P with FEDER funds (Spain).}

\begin{abstract}
We describe the local structure of self-dual gradient Ricci solitons in neutral signature. If the Ricci soliton is non-isotropic then it is locally conformally flat and locally isometric to a warped product of the form $I\times_\varphi N(c)$, where $N(c)$ is a space of constant curvature. If the Ricci soliton is isotropic, then it is locally isometric to the cotangent bundle of an affine surface equipped with the Riemannian extension of the connection, and the Ricci soliton is described by the underlying affine structure. This provides examples of self-dual gradient Ricci solitons which are not locally conformally flat.
\end{abstract}

\maketitle

\section{Introduction}
Let $(M,g)$ be a neutral signature four-dimensional pseudo-Riemannian manifold. Let $f\in\mathcal{C}^\infty(M)$, we say that the triple $(M,g,f)$ is
a \emph{gradient Ricci soliton} if
the following equation is satisfied:
\begin{equation}\label{eq:ricci-soliton}
\operatorname{Hes}_f+\rho=\lambda\, g\,,
\end{equation}
for some $\lambda\in\mathbb{R}$, where $\rho$ is the \emph{Ricci tensor}, and
$\operatorname{Hes}_f$ is the \emph{Hessian tensor} acting on $f$  defined by
\begin{equation}\label{eq:hessian}
\operatorname{Hes}_f(x,y)=(\nabla_x df)(y)=xy(f)-(\nabla_xy)(f)\,.
\end{equation}

%Let $\nabla f$ be the vector field
%dual to the exterior derivative $df$ of $f$.
%The {\it Hessian operator} $\begin{}}_f(X):=\nabla_X(\nabla f)$ satisfies $\operatorname{Hes}(X,Y)=g(\mathcal{H}_fX,Y)$.

%Let $R$ be the curvature tensor defined in terms of the Levi-Civita connection $\nabla$ by
%$R(X,Y)Z=\nabla_{[X,Y]} Z-\nabla_X\nabla_Y Z+\nabla_Y\nabla_X Z$.
%Denote by
%$\rho$  the {\it Ricci tensor} and by $\operatorname{Ric}$
%the {\it Ricci operator}; $\rho(X,Y)=g(\operatorname{Ric}X,Y)$.
%For $f\in C^\infty(M)$, let
%$\operatorname{Hes}_f$ be the Hessian tensor, which is defined by
%\begin{equation}\label{eq:hessian}
%\operatorname{Hes}_f(X,Y)=(\nabla_X df)(Y)=XY(f)-(\nabla_XY)(f)\,.
%\end{equation}
%Let $\nabla f$ be the vector field
%dual to the exterior derivative $df$ of $f$.
%The {\it Hessian operator} $\mathcal{H}_f(X):=\nabla_X(\nabla f)$ satisfies $\operatorname{Hes}(X,Y)=g(\mathcal{H}_fX,Y)$.
%%Note that
%%$\|\rho\|^2=\|\operatorname{Ric}\|^2$ and
%%$\|\mathcal{H}_f\|^2=\|\operatorname{Hes}_f\|^2$.

For $f$ constant, equation \eqref{eq:ricci-soliton} reduces to the Einstein equation $\rho=\lambda g$, so gradient Ricci solitons generalize Einstein manifolds in some way.
The main interest of gradient Ricci solitons comes from the fact that they correspond to self-similar solutions of the {\it Ricci flow} $\,\partial_tg(t)=-2\rho_{g(t)}$. If $\lambda>0$ or $\lambda<0$ then $(M,g,f)$ is said to be \emph{shrinking} or \emph{expanding}, respectively, whereas if $\lambda=0$ it is said to be \emph{steady}.
Gradient Ricci solitons have been extensively
investigated in the literature -- see for example the discussion in
\cite{BGG, Cao, Derd-rs} and the references therein.

Self-dual Einstein metrics are important not only in geometry but also in mathematical physics since they constitute the background structure of gravitational instantons. Furthermore note that those metrics are locally conformally Bochner-flat Kähler on the open set where $\|W^+\|^2\neq 0$, provided that $W^+$ has at most two-distinct eigenvalues \cite{Derd-3} in the Riemannian category.
Moreover, a geometrical interpretation of the Einstein self-dual condition in terms of the spectrum of the Jacobi operators was investigated as the first non-trivial case of Osserman manifolds (cf. \cite{BBR, DR-GR-VL, Derd}). We emphasize here that the geometry of (anti)-self-dual conformal structures is much richer in neutral signature $(2,2)$ than the corresponding Riemannian one (see, for example \cite{DW} and the references therein).

Self-dual gradient Ricci solitons in dimension four have been previously studied by Chen and Wang in  \cite{chen-wang}, assuming that the manifold has positive definite signature. They show that all non-trivial self-dual Riemannian gradient Ricci solitons are locally conformally flat.
In this paper we study self-dual gradient Ricci solitons in neutral signature.  
Our main result provides a local description, subject to the character of $\nabla f$. While in the non-isotropic case ($\|\nabla f\|^2\neq 0$) they are locally conformally flat, the isotropic case
($\|\nabla f\|^2= 0$) allows the existence of many examples which are not locally conformally flat, in contrast with the Riemannian setting. All isotropic self-dual gradient Ricci solitons are steady and they are locally realized as the cotangent bundle of suitable affine surfaces. 

Affine connections inducing gradient Ricci solitons are determined by an affine gradient Ricci soliton equation which involves the symmetric part of the Ricci tensor of the connection.
Let $(\Sigma, D)$ be an affine surface, let $R^D(X,Y)Z:=D_{[X,Y]} Z-D_XD_Y Z+D_YD_X Z$ denote the curvature tensor. The associated Ricci tensor $\rho^D$ decomposes into its symmetric and  anti-symmetric parts as follows:
\[
\rho^D_{sym}(X,Y)\!=\!\frac{1}{2}\left\{\rho^D(X,Y)+\rho^D(Y,X)\right\},\,\, \rho^D_{ant} (X,Y)\!=\!\frac{1}{2}\left\{\rho^D(X,Y)-\rho^D(Y,X)\right\}.
\]
Let $h\in\mathcal{C}^\infty(\Sigma)$, we say that $(\Sigma,D,h)$ is an \emph{affine gradient Ricci soliton} if the following equation is satisfied:
\begin{equation}\label{eq:affine-ricci-soliton}
\operatorname{Hes}^D_h+2\rho^D_{sym}=0\,,
\end{equation}
for some potential function $h\in\mathcal{C}^\infty(\Sigma)$, where $\operatorname{Hes}^D_h=Ddh$ is defined as in \eqref{eq:hessian}.

In order to state our results, we briefly recall some facts of the geometry of cotangent bundles.
Let $N$ be a manifold and $T^\ast N$ its cotangent bundle. We express any point $\xi\in T^*N$ as a pair $\xi=(p,\omega)$, with $\omega$  a one-form on $T_pN$. Thus $\pi:T^*N\rightarrow N$ defined by $\pi(p,\omega)=p$ is the natural projection. For any vector field $X$ on $N$, the evaluation map $\iota X$ is the smooth function on $T^*N$ defined by $\iota X(p,\omega)=\omega(X_p)$. The special significance of the evaluation map comes from the fact that vector fields on $T^*N$ are characterized by their action on evaluation maps (we refer to \cite{YI} for more information). Hence for a vector field $X$ on $N$ its  complete lift is the vector field on $T^*N$ defined by $X^C(\iota Z)=\iota [X,Z]$, for all vector fields $Z$ on $N$.
The crucial point is that tensor fields of type $(0,s)$ on $T^*N$ are determined by their action on complete lifts of vector fields on $N$.
Associated to any torsion-free connection $D$ on $N$, its \emph{Riemannian extension} is the neutral signature metric $g_D$ on $T^*N$ determined by the expression 
$g_D(X^C,Y^C)=-\iota(D_XY+D_YX)$.
The previous construction can be generalized by adding a deformation tensor as follows. Let $\Phi$ be a symmetric $(0,2)$-tensor field on $N$. The neutral signature metric on $T^*N$ defined by $g_{D,\Phi}=g_D+\pi^*\Phi$ is called a \emph{deformed Riemannian extension}.
Deformed Riemannian extensions are typical examples of Walker metrics since $\mathcal{D}=\operatorname{ker}\pi_*$ is a parallel degenerate plane field of maximum rank on $T^*N$. 

Since the Ricci soliton equation \eqref{eq:ricci-soliton} links the geometry of $(M,g)$, through its Ricci curvature, with that of the level sets of the potential function, by means of their second fundamental form, we divide our analysis in two parts: firstly we consider non-isotropic gradient Ricci solitons, this is, $\nabla f$ is a spacelike or a timelike vector field, and secondly we consider isotropic gradient Ricci solitons, this is, $\nabla f$ is a null vector field on a non-empty open subset of $M$.
The following result classifies four-dimensional neutral signature self-dual gradient Ricci solitons depending on the causal character of the gradient of the potential function.

\begin{theorem}\label{th:1}
Let $(M,g,f)$ be a non-trivial self-dual gradient Ricci soliton of neutral signature. 
\begin{enumerate}
\item[(1)] If $\|\nabla f\|\neq 0$ at $p$, then $(M,g)$ is isometric to a warped product $I\times_\varphi N$ where $N$ is a $3$-dimensional manifold of constant sectional curvature, in a neighbourhood of $p$.
Hence $(M,g)$ is locally conformally flat.

\item[(2)] If $\|\nabla f\|=0$ on an open subset $\mathcal{S}$ of $M$, then $(\mathcal{S},g)$ is locally isometric to the cotangent bundle $T^*\Sigma$ of an affine surface $(\Sigma,D)$ equipped with a deformed Riemannian extension $g_{D,\Phi}$.
Moreover, for $f\in\mathcal{C}^\infty(T^*\Sigma)$ and $h\in\mathcal{C}^\infty(\Sigma)$,  $(T^*\Sigma,g_{D,\Phi},f)$ is a gradient Ricci soliton if and only if $(\Sigma,D,h)$ is an affine gradient Ricci soliton where $f=\pi^\ast h$.
\end{enumerate}
\end{theorem}

Recall that the deformed Riemannian extension $(T^*\Sigma,g_{D,\Phi})$ of an affine surface is Einstein if and only if the symmetric part $\rho^D_{sym}$ of the Ricci tensor vanishes identically \cite{walker-metrics}. 

It is important to emphasize here that the deformation tensor $\Phi$ plays no role in the Ricci soliton equation. Therefore, for any affine gradient Ricci soliton $(\Sigma,D,h)$, all deformed Riemannian extensions $(T^*\Sigma,g_{D,\Phi},f=\pi^*h)$ are steady gradient Ricci solitons. This provides an infinite family of gradient Ricci solitons associated to each affine gradient Ricci soliton. Also note that, while the Riemannian extension $(T^*\Sigma,g_{D})$ is locally conformally flat if and only if $(\Sigma,D)$ is projectively flat, $(T^*\Sigma,g_{D,\Phi})$ is always self-dual but not locally conformally flat even if $(\Sigma,D)$ is projectively flat \cite{CLGRGVL, DR-GR-VL}. 
Moreover, all deformed Riemannian extensions in Theorem~\ref{th:1}-(2) have zero scalar curvature and, hence, their self-dual Weyl curvature operator $W^+$ is always nilpotent~\cite{DR-GR-VL}.

An application of Theorem \ref{th:1} and results in \cite{afifi} provide the following description of four-dimensional locally conformally flat gradient Ricci solitons with no prescribed signature. 

\begin{theorem}
\label{th:4D-lcf}
Let $(M,g,f)$ be a four-dimensional locally conformally flat gradient Ricci soliton. The following statements hold:
\begin{enumerate}
\item[(1)] If $\|\nabla f\|\neq 0$ at $p$, then $(M,g)$ is isometric to a warped product $I\times_\varphi N$ where $N$ is a $3$-dimensional manifold of constant sectional curvature, in a neighbourhood of $p$.
\item[(2)] If $\|\nabla f\|=0$ on an open subset $\mathcal{S}$ of $M$ and the metric is of Lorentzian signature, then $(\mathcal{S},g)$ is locally isometric to a locally conformally flat plane wave, i.e., there are coordinates $(u,v,x^1,x^2)$ such that the metric is given by
$g=2 dudv +a(u)\sum (x^i)^2 du^2 +\sum (dx^i)^2$ and $f(u,v,x_1,x_2)=f(u)$, with $f''(u)=2\,a(u)$.
\item[(3)] If $\|\nabla f\|=0$ on an open subset $\mathcal{S}$ of $M$ and the metric is of neutral signature, then $(\mathcal{S},g)$ is locally isometric to the cotangent bundle $T^*\Sigma$ of a projectively flat affine gradient Ricci soliton surface $(\Sigma,D,h)$ equipped with a Riemannian extension $g_{D}$ and $f=\pi^\ast h$.
\end{enumerate}
\end{theorem}

The paper is organized as follows. We state the basic curvature identities to be used through the paper in Section \ref{se-2}. Then Section~\ref{se-3} is devoted to investigate non-isotropic gradient Ricci solitons and to prove Theorem~\ref{th:1}-(1). The isotropic case is studied in Section \ref{se-4}, where we show that any non-trivial isotropic self-dual gradient Ricci soliton is steady and that it is realized as the cotangent bundle of an affine surface endowed with a deformed Riemannian extension. The proof of Theorem \ref{th:1}-(2) is then completed in Section \ref{se-5}, where affine gradient Ricci solitons are discussed. Finally, some examples are investigated in Section \ref{se-6}, with special attention to the existence of non-trivial affine gradient Ricci solitons on projectively flat homogeneous affine surfaces.

\section{Basic formulas and self-dual gradient Ricci solitons}\label{se-2}

We summarize in the following lemma some known formulas which hold in general for gradient Ricci solitons (see, for example \cite{BGG,Cao, RicciFlow,E-LN-M}).

\begin{lemma}\label{lemma:grs-formulas}
Let $(M,g,f)$ be a gradient Ricci soliton.
Then the following relations hold:
\begin{enumerate}
\item $\nabla \tau=2\operatorname{Ric}(\nabla f)$\,,
\item $\tau+\|\nabla f\|^2-2\lambda f=\operatorname{const}$\,,
\item $R(x,y,z,\nabla f)=-(\nabla_x \rho)(y,z)+(\nabla_y \rho)(x,z)$\,.
%\item $\left(\nabla_{\nabla f}\operatorname{Ric}\right)
%+\operatorname{Ric}\circ \mathcal{H}_f=-R(\nabla f,\cdot)\nabla f+\frac12 \nabla\nabla \tau$\,,
\end{enumerate}
%where $\mathcal{H}_f(X)=\nabla_X(\nabla f)$ denotes the Hessian operator.
\end{lemma}

The Weyl tensor in dimension $4$ for an arbitrary pseudo-Riemannian manifold $(M,g)$ is given by
\begin{equation}\label{eq:weyl}
\hspace*{-0.5cm}
\begin{array}{rcl}
    W(x,y,z,v) & = & R(x,y,z,v) + \frac{\tau}{6} \{g(x,z)g(y,v)-g(y,z)g(x,v)\}
    \vspace{0.1in}\\
    &&\qquad- \frac{1}{2} \{\rho(x,z)g(y,v)-\rho(y,z)g(x,v)
    \vspace{0.1in}\\
    && \qquad\hspace*{1.1cm}+ \rho(y,v)g(x,z)-\rho(x,v)g(y,z)\}.
\end{array}
\end{equation}
Now, using $(3)$ and $(1)$ in Lemma~\ref{lemma:grs-formulas} one can write the Weyl tensor where one of the arguments is $\nabla f$ as
\begin{equation}\label{eq:weyl2}
\hspace*{-0.5cm}
\begin{array}{rcl}
    W(x,y,z,\nabla f) & = & - C(x,y,z) + \frac{\tau}{6} \{g(x,z)g(y,\nabla f)-g(y,z)g(x,\nabla f)\}
    \vspace{0.1in}\\
    &&\qquad- \frac{1}{2} \{\rho(x,z)g(y,\nabla f)-\rho(y,z)g(x,\nabla f)\}
    \vspace{0.1in}\\
    && \qquad\hspace*{1.1cm}-\frac{1}{6}\{ \rho(y,\nabla f)g(x,z)-\rho(x,\nabla f)g(y,z)\},
\end{array}
\end{equation}
where $C(x,y,z)=(\nabla_x\rho)(y,z)-(\nabla_y\rho)(x,z)-\displaystyle\frac{1}{6}\{x(\tau)g(y,z)-y(\tau)g(x,z)\}$ is the Cotton tensor.

\medskip
Let $\{e_1,e_2,e_3,e_4\}$ be an orthonormal local frame. The self-dual $(\Lambda^+)$ and anti-self-dual $(\Lambda^-)$ spaces of two-forms are generated by
\[
\Lambda^{\pm}=\operatorname{span}\{e^1\wedge e^2\pm \varepsilon_3\varepsilon_4 e^3\wedge e^4,e^1\wedge e^3\mp \varepsilon_2\varepsilon_4 e^2\wedge e^4,e^1\wedge e^4\pm \varepsilon_2\varepsilon_3 e^2\wedge e^3\}.
\]
where $\varepsilon_i=g(e_i,e_i)$ for $i\in\{2,3,4\}$.
Assume the manifold is self-dual, this is $W^-=0$. We consider subindeces $i$, $j$ and $k$ so that $\{i,j,k\}=\{2,3,4\}$. Denote by $\sigma_{ijk}$ the sign of the corresponding permutation. Then self-duality is characterized by
\begin{equation}\label{eq:self-dual}
W(e_1,e_i,z,T)=\sigma_{ijk} \varepsilon_j\varepsilon_k W(e_j,e_k,z,T) \quad\text{ for any vector fields } z \text{ and } T.
\end{equation}

Now, since the Cotton tensor is  the {divergence} of the Weyl tensor:  $C(x,y,z)=(\operatorname{div} W)(x,y,z,\cdot)$, we use expression \eqref{eq:weyl2} to write equation \eqref{eq:self-dual} as (we write the expression as a $(1,3)$-tensor equation for simplicity):

\begin{equation}\label{eq:relation1}
\begin{array}{l}
\displaystyle
\tau\{g(e_i,\nabla f)e_1-g(e_1,\nabla f)e_i\}
- \{\rho(e_i,\nabla f)e_1-\rho(e_1,\nabla f)e_i
\\
\noalign{\medskip}
\displaystyle
\phantom{\tau \{g(e_i,\nabla f)e_1-\}}
+3g(e_i,\nabla f)\operatorname{Ric}(e_1)-3g(e_1,\nabla f)\operatorname{Ric}(e_i)\}\\
\noalign{\medskip}
\phantom{\frac{\tau}{6} \{}
\displaystyle
=\sigma_{ijk}\, \varepsilon_j\varepsilon_k\big(  
 \tau \{g(e_k,\nabla f)e_j-g(e_j,\nabla f)e_k\}
-\{\rho(e_k,\nabla f)e_j-\rho(e_j,\nabla f)e_k
\\
\noalign{\medskip}
\displaystyle 
\phantom{\tau\{g(e_i,\nabla f)e_1-\}} 
  +3 g(e_k,\nabla f)\operatorname{Ric}(e_j)-3 g(e_j,\nabla f)\operatorname{Ric}(e_k)\}
\big)
\end{array}
\end{equation}

%{\red 
%\begin{equation}\label{eq:relation1}
%\begin{array}{l}
%\displaystyle
%\frac{\tau}{6} \{g(e_i,\nabla f)e_1\!-\!g(e_1,\nabla f)e_i\}
%- \frac{1}{6} \{\rho(e_i,\nabla f)e_1\!-\!\rho(e_1,\nabla f)e_i
%\\
%\noalign{\medskip}
%\displaystyle
%\phantom{\frac{\tau}{6} \{g(e_i,\nabla f)e_1\!-\!\}}
%\!+\!3g(e_i,\nabla f)\operatorname{Ric}(e_1)\!-\!3g(e_1,\nabla f)\operatorname{Ric}(e_i)\}\\
%\noalign{\medskip}
%\phantom{\frac{\tau}{6} \{}
%\displaystyle
%=\sigma_{ijk}\, \varepsilon_j\varepsilon_k\!\bigg(\!  
% \frac{\tau}{6} \{g(e_k,\nabla f)e_j\!-\!g(e_j,\nabla f)e_k\}
%\!-\!\frac{1}{6} \{\rho(e_k,\nabla f)e_j\!-\!\rho(e_j,\nabla f)e_k
%\right.\\
%\noalign{\medskip}
%\displaystyle
%\phantom{\frac{\tau}{6} \{g(e_i,\nabla f)e_1\!-\!\}} 
%\left. 
%  \!+\!3g(e_k,\nabla f)\operatorname{Ric}(e_j)\!-\!3g(e_j,\nabla f)\operatorname{Ric}(e_k)\}
%\bigg)
%\end{array}
%\end{equation}
%}

\section{Non-isotropic self-dual gradient Ricci solitons}\label{se-3}

In order to prove Theorem~\ref{th:1}-(1), in this section we study  non-isotropic gradient Ricci solitons. We begin by analyzing the diagonalizability of the Ricci operator $\operatorname{Ric}$
associated to the Ricci tensor ($\rho(x,y)=g(\operatorname{Ric}x,y)$).

\begin{lemma}\label{lemma:non-isotropic-Ricdiagonalizable}
Let $(M,g,f)$ be a non-trivial self-dual gradient Ricci soliton of neutral signature. If the potential function $f$ satisfies $\|\nabla f\|^2\neq 0$, then $\operatorname{Ric}$ is diagonalizable and $\nabla f$ is an eigenvector of $\operatorname{Ric}$. 
\end{lemma}
\begin{proof}
If $\nabla f$ is non-null, this is, the gradient Ricci soliton is non-isotropic, assume without loss of generality that $e_1=\frac{\nabla f}{\|\nabla f\|}$ and $\nabla f\perp e_i,e_j,e_k$. Then \eqref{eq:relation1} reduces to:
\begin{equation}\label{eq:lemm4-expression}
\begin{array}{l}
\tau g(e_i,z)g(e_1,\nabla f)- 3\rho(e_i,z)g(e_1,\nabla f) \rho(e_i,\nabla f)g(e_1,z)-\rho(e_1,\nabla f)g(e_i,z)
     \\
\noalign{\medskip}
\phantom{\frac{\tau}{6} g(e_i,z)g(e_1,\nabla f)}     
     =\sigma_{ijk} \varepsilon_j\varepsilon_k\{ \rho(e_k,\nabla f)g(e_j,z)-\rho(e_j,\nabla f)g(e_k,z)\}\,,
\end{array}
\end{equation}
for an arbitrary vector field $z$.
Now, set $z=e_1$ in \eqref{eq:lemm4-expression} to get
\[
    3 \rho(e_i,e_1)g(e_1,\nabla f)
      - \rho(e_i,\nabla f)\varepsilon_1
      = 0,
\]
from where $\rho(e_i,\nabla f)=0$ for $i=2,3,4$. Hence $\nabla f$ is an eigenvector of $\operatorname{Ric}$.

Next, set $z=e_i$ in \eqref{eq:lemm4-expression} so that
\[
\tau \varepsilon_i g(e_1,\nabla f)
    - 3 \rho(e_i,e_i)g(e_1,\nabla f)
    -\rho(e_1,\nabla f)\varepsilon_i
    =0\,,
\]
and thus (since $\|\nabla f\|^2\neq 0$)
\[
\tau
    - 3 \rho(e_i,e_i)\varepsilon_i
    -\rho(e_1,e_1)\varepsilon_1
    =0\,.
\]
Hence $\rho(e_2,e_2)\varepsilon_2=\rho(e_3,e_3)\varepsilon_3=\rho(e_4,e_4)\varepsilon_4$.

Finally, set $z=e_j$ in \eqref{eq:lemm4-expression} to obtain
\[
    \rho(e_i,e_j)g(e_1,\nabla f) = 0\,,
\]
and thus, since $g(e_1,\nabla f)\neq 0$, one has that $\rho(e_i,e_j)=0$ for all $i,j=2,3,4$, $i\neq j$, which shows that $Ric$ is diagonalizable.
\end{proof}

Now we have enough information on the spectrum of $\operatorname{Ric}$ to show that a self-dual gradient Ricci soliton with $\|\nabla f\|^2\neq 0$ decomposes locally as a warped product of the form $I\times_\varphi N$, where $N$ has constant sectional curvature.
\begin{proof}[Proof of Theorem~\ref{th:1}-(1).]
Adopt the notation in the proof of Lemma~\ref{lemma:non-isotropic-Ricdiagonalizable}. Note that $\operatorname{Ric}$ has  eigenvalues $\mu=\rho(e_1,e_1)\varepsilon_1$ and $\nu=\rho(e_i,e_i)\varepsilon_i=\frac{\tau-\mu}3$, ($i=2,3,4$). Hence, for $i=2,3,4$, the Ricci soliton equation \eqref{eq:ricci-soliton} shows that 
\[
Hes_f(e_i,e_i)=\lambda g(e_i,e_i)-\rho(e_i,e_i)=\left(\lambda-\frac{\tau-\mu}3\right) g(e_i,e_i),
\]
and thus the level sets of $f$ are totally umbilical hypersurfaces. 
Since the one-dimensional distribution $\operatorname{span}\{\nabla f\}$ is totally geodesic one has that $(M,g)$ decomposes locally as a twisted product of the form $I\times_\varphi N$ (see \cite{ponge-reckziegel}). Moreover $\rho(e_1,e_i)=0$ ($i=2,3,4$) shows that the twisted product reduces to a warped product \cite{manolo-eduardo}. Finally, since $I\times_\varphi N$ is self-dual, it is necessarily locally conformally flat and the fiber $N$ is of constant sectional curvature (see \cite{bv-gr-vl-2}).
\end{proof}

\begin{remark}\rm
The potential function $f$ in Theorem~\ref{th:1}-(1) is a radial function $f(t)$, and hence a direct computation from the soliton equation \eqref{eq:ricci-soliton} shows that it is given as a solution to the equations:
\begin{eqnarray}\label{eq:non-isotropic1}
f^{\prime \prime}&=&\varepsilon\,\lambda + 3\frac{\varphi^{\prime \prime}}{\varphi}, \\
\label{eq:non-isotropic2}
 \varphi\,\varphi^\prime f^\prime&=&\varepsilon\,\lambda\, \varphi^2 -2\,\varepsilon\,c+\varphi\,\varphi^{\prime \prime}+2(\varphi^\prime)^2,
\end{eqnarray}
where $\varepsilon=1$ if $\nabla f$ is spacelike and $\varepsilon=-1$ if $\nabla f$ is timelike. Thus, differentiating equation \eqref{eq:non-isotropic2} we see that the warping function $\varphi$ is not arbitrary, but a solution of
\begin{equation}\label{eq:ode1}
-2(\varphi^\prime )^4+2c\varepsilon((\varphi^\prime)^2+ \varphi\varphi^{\prime\prime})-\varphi(\varphi^\prime)\varphi^{\prime\prime}-\varepsilon\lambda \varphi^3\varphi^{\prime\prime}-(\varphi\varphi^{\prime\prime})^2+\varphi^2 \varphi^\prime \varphi^{\prime\prime\prime}=0\,.
\end{equation}
Moreover, the potential function of the soliton $f$ is determined from $\varphi$ up to a constant. Thus, on a manifold $I\times_\varphi N$ where $\varphi$ satisfies \eqref{eq:ode1}, the potential function always exists locally and is essentially unique. This is justified by the fact that a warped product of the form $I\times_\varphi N$ does not admit nontrivial homothetic vector fields unless it is Ricci flat. 
\end{remark}

Non-trivial Riemannian self-dual gradient Ricci solitons are locally conformally flat as shown in \cite{chen-wang}. Note that Theorem \ref{th:1}-(1) covers the Riemannian situation. In what follows we study the strictly non-Riemannian case when $\nabla f$ is an isotropic vector field. This analysis leads to new examples without Riemannian counterpart.

\bigskip

\section{Isotropic self-dual gradient Ricci solitons}\label{se-4}

We devote this section to analyze isotropic self-dual gradient Ricci solitons in dimension four and neutral signature. We assume henceforth that the gradient Ricci soliton is non-trivial. Since $\nabla f$ is nonzero and null, there exist a unit spacelike vector field $e_1$ and a unit timelike vector field $e_2$ such that $\nabla f=\frac{e_1+e_2}2$. We complete these set of vector fields to a local frame $\{e_1(-),e_2(+),e_3(-),e_4(+)\}$ (where $e_i(\pm)$ indicates the causal character of $e_i$) and build the following new one:
\[
\mathcal{B}=\{\nabla f=\frac{e_1+e_2}2, u=\frac{-e_3+e_4}2,v=\frac{-e_1+e_2}2,w=\frac{e_3+e_4}2\}\,.
\]
All vector fields $\nabla f,u,v,w$ are null and the only nonzero components of the metric tensor expressed in the local frame $\mathcal{B}$ are
\[g(\nabla f,v)=g(u,w)=1.\]
Furthermore, the self-duality condition in equation \eqref{eq:self-dual} expresses in terms of the vector fields of $\mathcal{B}$ as:
\begin{eqnarray}\label{eq:self-dual-basisB1}
W(\nabla f,v,z,t)&=&W(u,w,z,t)\,,\\
\noalign{\medskip}\label{eq:self-dual-basisB2}
W(u,v,z,t)&=&0\,,\\
\noalign{\medskip}\label{eq:self-dual-basisB3}
W(\nabla f,w,z,t)&=&0\,.
\end{eqnarray}
for all vector fields $z,t$.

\begin{lemma}\label{lemma:isotropic-steady}
Let $(M,g,f)$ be an isotropic self-dual gradient Ricci soliton. Then $(M,g,f)$ is steady, $\tau=0$ and $\operatorname{Ric}(\nabla f)=0$.
\end{lemma}
\begin{proof}
Note from Lemma~\ref{lemma:grs-formulas} that
\[
2\operatorname{Ric}(\nabla f)=\nabla \tau=2\lambda \nabla f\,,
\]
so $\nabla f$ is an eigenvector for the Ricci operator asssociated to the eigenvalue $\lambda$.

From equation \eqref{eq:self-dual-basisB1} we have that $W(\nabla f,v,z,\nabla f)=W(u,w,z,\nabla f)$ and moreover that $C(\nabla f,v,z)=C(u,w,z)$, hence we use \eqref{eq:weyl2} to write
\[
0=W(\nabla f,v,z,\nabla f)-W(u,w,z,\nabla f)=\left(\frac{\tau}{6}-\frac{2}{3}\lambda\right) g(\nabla f,z)\,.
\]

Hence $\tau=4\lambda$ and $\tau$ is constant. But for $\tau$ constant, we get that $0=2\operatorname{Ric}(\nabla f)=2\lambda f$, so $\lambda=\tau=0$ and the gradient Ricci soliton is necessarily steady.
\end{proof}

\begin{lemma}\label{lemma:isotropic-nulldistributioneigenspace}
Let $(M,g,f)$ be an isotropic self-dual gradient Ricci soliton. Then there exists a $2$-dimensional null distribution $\mathcal{D}$ such that $\operatorname{Ric}(\mathcal{D})=0$ and the Ricci operator is two-step-nilpotent.
\end{lemma}
\begin{proof}
From equation \eqref{eq:self-dual-basisB2} we have that $W(u,v,z,t)=0$ and, hence, $C(u,v,z)=0$. Now, from equation \eqref{eq:weyl2} we compute
\[
0=W(u,v,z,\nabla f)=\left(\frac\tau{6}-\frac{\lambda}{6}\right)g(u,z)-\frac{1}2\rho(u,z)\,.
\]
Since $\tau=\lambda=0$ from Lemma~\ref{lemma:isotropic-steady}, we get that $\rho(u,z)=0$ for any vector field $z$. Hence $\operatorname{Ric}(u)=0$.
Define the distribution $\mathcal{D}=\operatorname{span}\{\nabla f,u\}$ and note that $\operatorname{Ric}(\mathcal{D})=0$.
Moreover, as a consequence of this fact, the only possibly nonzero components of the Ricci tensor are $\rho(v,v)=a$, $\rho(w,w)=b$ and $\rho(v,w)=c$. Hence, the matrix associated to the Ricci operator in the local frame $\mathcal{B}$ becomes
\begin{equation}\label{eq:ricci-matrix}
\operatorname{Ric}=\left(\begin{array}{cccc}
0&0&a&c\\
0&0&c&b\\
0&0&0&0\\
0&0&0&0
\end{array}\right).
\end{equation}
Thus, $\operatorname{Ric}^2=0$.
\end{proof}

\begin{remark}
\rm
The distribution $\mathcal{D}$ in Lemma~\ref{lemma:isotropic-nulldistributioneigenspace} is self-dual and is well-defined. Associated to the null vector $\nabla f$, consider the $2$-dimensional
non-degenerate cocient subspace $\overline{\nabla f}^\perp=(\nabla f)^\perp/\operatorname{span}\{\nabla f\}$, which inherits a Lorentzian metric. Let $\overline{u}\in \overline{\nabla f}^\perp$ be a null vector field such that the $2$-plane $\sigma=\operatorname{span}\{\nabla f,u\}$ is self-dual (i.e, $\nabla f\wedge u$ is a self-dual $2$-form), where $u$ is an arbitrary representative on $\nabla f^\perp$ of $\overline{u}$. Observe that the distribution $\mathcal{D}$ coincides with $\sigma$ and is uniquely determined.

It is worth emphasizing that all our results are local so we may change self-duality by anti-self-duality in our analysis. However this choice of orientation determines the Walker structure of the manifold and viceversa.
\end{remark}

As a consequence of Lemma~\ref{lemma:isotropic-nulldistributioneigenspace} the image of $\operatorname{Ric}$ is totally isotropic:  $\operatorname{Im}(\operatorname{Ric})\subset \mathcal{D}$. Moreover, since the Ricci soliton is steady, the Hessian operator of $f$ defined by $\operatorname{hes}_f(x)=\nabla_x\nabla f$ satisfies $\operatorname{hes}_f=-\operatorname{Ric}$, and hence $\nabla_x \nabla f\in \mathcal{D}$ for all $x$. Also, we obtain that $\operatorname{Ric}\circ \operatorname{hes}_f=0$ and therefore $(\nabla_x \operatorname{Ric})(\nabla f)=0$.  From Lemma~\ref{lemma:grs-formulas} $R(x,y)\nabla f=(\nabla_x \operatorname{Ric})(y)-(\nabla_y \operatorname{Ric})(x)$ and then
\[
R(\nabla f,x)\nabla f=\left(\nabla_{\nabla f}\operatorname{Ric}\right)(x)
.
\]
We already know that the distribution $\mathcal{D}$ is totally isotropic, in the next lemma we show that it is also parallel and therefore $(M,g)$ is a Walker manifold.
\begin{lemma}\label{lemma:isotropic-walker}
Let $(M,g,f)$ be an isotropic self-dual gradient Ricci soliton of dimension $4$. Then $(M,g)$ is a Walker manifold.
\end{lemma}
\begin{proof}
We see that $\mathcal{D}=\operatorname{span}\{\nabla f,u\}$ is a parallel distribution as follows.
Since $\operatorname{Im}(\operatorname{hes}_f)\subset \mathcal{D}$, we have that $\nabla_x \nabla f\in \mathcal{D}$ for all $x$.
Since $\mathcal{D}$ is totally isotropic and $\mathcal{D}^\perp=\mathcal{D}$ we see that $\nabla_x u\in \mathcal{D}$ as follows:
\[
0=x\, g(u,u)=2 g(\nabla_x u,u)\,,
\]
\[
0=x\,g(u,\nabla f)=g(\nabla_xu,\nabla f)+g(u,\nabla_x \nabla f)=g(\nabla_xu,\nabla f)\,.
\]
Therefore $\nabla_x\mathcal{D}\subset\mathcal{D}$ for all vector fields $x$, which shows that $\mathcal{D}$ is parallel.
\end{proof}

Deformed Riemannian extensions are characterized among Walker manifolds by a curvature condition which is, in a certain sense, analogous to the curvature condition characterizing $pp$-waves among Brinkmann waves in Lorentzian signature~\cite{Leistner}.

\begin{theorem}{\rm \cite{afifi}}\label{th:riemann-extensions-characterization}
Let $(M,g)$ be a four-dimensional manifold admitting a $2$-dimen\-sion\-al parallel null distribution $\mathcal{D}$. Then $g$ is the deformed Riemannnian extension of an affine manifold if and only if
$R(x,\mathcal{D})\mathcal{D}=0$
for all vector fields $x$ on $M$.
\end{theorem}

Now the first statement of Theorem~\ref{th:1}-(2) is proven as follows.

\begin{lemma}
\label{lema-nuevo-10}
Let $(M,g,f)$ be an isotropic non-trivial gradient Ricci soliton of dimension four. Then $(M,g)$ is the cotangent bundle of an affine surface $(\Sigma,D)$ equipped with the deformed Riemannian extension $g_{D,\Phi}$, where $\Phi$ is a symmetric $(0,2)$-tensor field on $\Sigma$.
\end{lemma}

\begin{proof}
As a consequence of Lemmas~\ref{lemma:isotropic-steady}, \ref{lemma:isotropic-nulldistributioneigenspace} and \ref{lemma:isotropic-walker} we know that $(M,g)$ is a Walker manifold. We work with the previously constructed local frame $\mathcal{B}$  to prove that $R(x,\mathcal{D})\mathcal{D}=0$ for any vector field $x$ and obtain the result as a consequence of Theorem~\ref{th:riemann-extensions-characterization}. More specifically,  we are going to see that $R(\nabla f,x,\nabla f,x)=R(\nabla f,x,u,y)=R(u,x,u,y)=0$ for any vector fields $x,y\in\{\nabla f, u, v,w\}$.

From expression \eqref{eq:weyl} we get that
\[
\begin{array}{rcl}
W(\nabla f,v,z,t)&=&R(\nabla f,v,z,t)+\frac{c}2\{g(u,z)g(\nabla f,t)-g(u,t)g(\nabla f,z)\}\,,\\
\noalign{\bigskip}
W(u,w,z,t)&=&R(u,w,z,t)+\frac{c}2\{g(\nabla f,z)g(u,t)-g(\nabla f,t)g(u,z)\}\,,\\
\noalign{\bigskip}
W(u,v,z,t)&=&R(u,v,z,t)+\frac{a}2\{g(\nabla f,z)g(u,t)-g(\nabla f,t)g(u,z)\}\,,\\
\noalign{\bigskip}
W(\nabla f,w,z,t)&=&R(\nabla f,w,z,t)+\frac{b}2\{g(u,z)g(\nabla f,t)-g(u,t)g(\nabla f,z)\}\,,
\end{array}
\]
where $a$, $b$ and $c$ are the function coefficients in \eqref{eq:ricci-matrix}.
From these expressions and equations \eqref{eq:self-dual-basisB1}, \eqref{eq:self-dual-basisB2} and \eqref{eq:self-dual-basisB3} we see that $(M,g)$ is self-dual if and only if the following relations hold for any vector fields $z$ and $t$:
\begin{equation}\label{eq:self-dual-basisB}
\begin{array}{c}
R(\nabla f,v,z,t)-R(u,w,z,t)+c\{g(u,z)g(\nabla f,t)-g(u,t)g(\nabla f,z)\}=0\,,\\
\noalign{\bigskip}
R(u,v,z,t)+\frac{a}2\{g(\nabla f,z)g(u,t)-g(\nabla f,t)g(u,z)\}=0\,,\\
\noalign{\bigskip}
R(\nabla f,w,z,t)+\frac{b}2\{g(u,z)g(\nabla f,t)-g(u,t)g(\nabla f,z)\}=0\,.
\end{array}
\end{equation}

We use these expressions to check directly that the following terms of the curvature vanish:
\begin{equation}\label{eq:0curvatureterms}
\begin{array}{l}
R(\nabla f,w,\nabla f,t)=0\text{ for any vector field } t,\\
\noalign{\medskip}
R(u,v,u,v)=0,\qquad R(u,v,u,w)=0\\
\noalign{\medskip}
R(u,v,\nabla f,v)=0,\qquad R(u,w,\nabla f,w)=0,\qquad R(u,v,\nabla f,w)=0.
\end{array}
\end{equation}

We begin by seeing that $R(\nabla f,x,\nabla f,y)=0$ for $x,y\in\{\nabla f, u, v,w\}$. So we compute:
\[
\begin{array}{rcl}
R(\nabla f,u,\nabla f,u)&=&g(\nabla_{[\nabla f,u]}\nabla f,u)-g(\nabla_{\nabla f}\nabla_u\nabla f-\nabla_u\nabla_{\nabla f}\nabla f,u)\\
\noalign{\medskip}
&=&g(\nabla_{[\nabla f,u]}\nabla f,u)= g(\nabla_u\nabla f,[\nabla f,u])=0\,.\\
\noalign{\bigskip}
R(\nabla f,u,\nabla f,v)&=&g(\nabla_{[\nabla f,u]}\nabla f,v)-g(\nabla_{\nabla f}\nabla_u\nabla f-\nabla_u\nabla_{\nabla f}\nabla f,v)\\
\noalign{\medskip}
&=& g(\nabla_{[\nabla f,u]}\nabla f,v)= g(\nabla_v\nabla f,[\nabla f,u])\\
\noalign{\medskip}
&=& g(-c\,\nabla f-b\,u,\nabla_{\nabla f}u)\\
\noalign{\medskip}
&=&-c\,g(\nabla f,\nabla_{\nabla f} u)-b\, g(u,\nabla_{\nabla f}u)\\
\noalign{\medskip}
&=&-c\, \nabla f\,g(\nabla f,u)+c\,g(\nabla_{\nabla f} \nabla f,u)-\frac{b}2\,\nabla f\,g(u,u)=0\,.
\end{array}
\]
Since 
\[
0=\rho(\nabla f,v)=R(\nabla f,u,v,w)+R(\nabla f,w,v,u)+R(\nabla f,v,v,\nabla f),
\] 
but $R(\nabla f,w,v,u)=0$ by \eqref{eq:0curvatureterms}, we have that
\[
\begin{array}{rcl}
R(\nabla f,v,\nabla f,v)&=& R(\nabla f,u,v,w)=R(v,w,\nabla f,u)\\
\noalign{\medskip}
&=& g((\nabla_{[v,w]}-\nabla_v\nabla_w+\nabla_w\nabla_v) \nabla f,u)\\
\noalign{\medskip}
&=& -g(\nabla_v\nabla_w\nabla f,u)+g(\nabla_w\nabla_v \nabla f,u)\\
\noalign{\medskip}
&=& g(\nabla_v(c\nabla f+b\, u),u)-g(\nabla_w(a\nabla f+cu),u)\\
\noalign{\medskip}
&=& v(c)g(\nabla f,u) +c\,g(\nabla_v\nabla f,u)+v(b) g(u,u)+b\,g(\nabla_vu,u)\\
\noalign{\smallskip}
&&-w(a)g(\nabla f,u)-a\,g(\nabla_w\nabla f,u)-w(c)g(u,u)-c\,g(\nabla_wu,u)\\
\noalign{\medskip}
&=&0\,.
\end{array}
\]
Thus we get that $R(\nabla f,x,\nabla f,y)=0$ for any $x,y\in\{\nabla f, u, v,w\}$.
On the one hand we have that $R(u,\nabla f,u,v)=\frac{1}2\rho(u,u)=0$ and, on the other hand, by \eqref{eq:self-dual-basisB}, we have that $R(u,\nabla f,u,w)=R(u,\nabla f,\nabla f,v)$ and that $R(u,w,u,w)=R(\nabla f,v,\nabla f,v)$. This, together with \eqref{eq:0curvatureterms}, implies that  $R(u,x,u,y)=0$ for any $x,y\in\{\nabla f, u, v,w\}$.

Finally, since $R(u,w,\nabla f,v)=R(\nabla f,v,\nabla f,v)$ by \eqref{eq:self-dual-basisB}, the terms computed above show that $R(u,x,\nabla f,y)=0$ for any $x,y\in\{\nabla f, u, v,w\}$, and thus $R(z,\mathcal{D})\mathcal{D}=0$ for arbitrary $z$.
\end{proof}

\section{Gradient Ricci solitons on Riemann Extensions. The proof of Theorem~\ref{th:1}-(2).}\label{se-5}

Since the underlying structure of any self-dual isotropic gradient Ricci soliton is that of a deformed Riemannian extension, we analyze the existence of gradient Ricci solitons on the cotangent bundle $T^*\Sigma$ of an affine surface $(\Sigma,D)$ equipped with the deformed Riemann extension $g_{D,\Phi}=g_D+\pi^*\Phi$, where $\Phi$ is a symmetric $(0,2)$-tensor field on $\Sigma$. %Our main result is a reduction result to the affine setting by showing that the potential function $f$ of a gradient Ricci soliton on $T^*\Sigma,g_{D,\Phi})$ reduces to the pullback $f=\pi^*h$ of a smooth function $h$ on $\Sigma$ which is a solution of the affine gradient Ricci soliton equation
%$\operatorname{Hes}^D_{h}+2\rho^D_{sym}=0$ on $(\Sigma,D)$, where  $\operatorname{Hes}^D$ is the Hessian tensor defined by the affine connection $D$ ($\operatorname{Hes}^D_{h}=D d h$) and $\rho^D_{symm}$ is the symmetric part of the Ricci tensor of the affine connection $D$.
This analysis ends up completing the proof of Theorem~\ref{th:1}-(2).

Let $(x^1,x^2)$ be local coordinates on $\Sigma$ and let $(x^1,x^2,x_{1'},x_{2'})$ be the induced coordinates on $T^*\Sigma$. Further, let ${}^D\Gamma_{ij}^k$ denote the Chrisfoffel symbols of the connection $D$ and let $\Phi_{ij}$ be the component functions of $\Phi$ in the coordinates $(x^1,x^2)$. Then the deformed Riemannian extension $g_{D,\Phi}=g_D+\pi^*\Phi$ expresses in the coordinates $(x^1,x^2,x_{1'},x_{2'})$ as
\begin{equation}
\label{eq:5-1}
g_{D,\Phi}=  dx^i\otimes dx_{i'}+dx_{i'}\otimes dx^i+\{\Phi_{ij}(x) -2x_{k'}{}^D\Gamma_{ij}{}^k\} dx^i\otimes dx^j\,.
\end{equation}

In order to simplify notation, we denote $\partial_{x^i}:=\frac{\partial}{\partial x^i}$, $\partial_{x_{i'}}:=\frac{\partial}{\partial x_{i'}}$ and $\partial^2_{x_{i'}x_{j'}}:=\frac{\partial^2}{\partial x_{i'}\partial x_{j'}}$ henceforth. The Levi-Civita connection of $g_{D,\Phi}$ is determined by the following possibly nonzero Christoffel symbols
\[
\begin{array}{l}
{\Gamma}_{ij}^k={}^D\Gamma_{ij}^k \, ,
\qquad\qquad
{\Gamma}_{{i'}j}^{{k'}}=-{}^D\Gamma_{jk}^i \, ,
\qquad\qquad
{\Gamma}_{i{j'}}^{{k'}}=-{}^D\Gamma_{ik}^j \, , \,\\
\noalign{\medskip}
\Gamma_{ij}^{k'}=\displaystyle
     \sum_{r=1}^2  x_{r'} \left(
    \partial_{x^k} {}^D\Gamma_{ij}^r - \partial_{x^i} {}^D\Gamma_{jk}^r - \partial_{x^j} {}^D\Gamma_{ik}^r
    + 2 \sum_{l=1}^2  {}^D\Gamma_{kl}^r {}^D\Gamma_{ij}^l
    \right)\\
\noalign{\medskip}
\phantom{\Gamma_{ij}^{k'}=}\displaystyle
    + \frac{1}{2}\left( \partial_{x^i} \Phi_{jk} + \partial_{x^j} \Phi_{ik} - \partial_{x^k} \Phi_{ij}\right)
    - \sum_{l=1}^2 \Phi_{kl} {}^D\Gamma_{ij}^l,
\end{array}
\]
where $i,j,k\in \{1,2\}$.

Now, a straightforward calculation shows that the only possibly nonzero components of the Ricci tensor of $g_{D,\Phi}$ are given by
%\begin{equation}\quad
%\label{eq:5-2}
%\begin{array}{l}
%\rho(\partial_{x^1},\partial_{x^1})
%\!=\!2\rho_D^{sym}(\partial_{x^1},\partial_{x^1})\\
%\noalign{\medskip}
%\phantom{\rho(\partial_{x^1},\partial_{x^1})}
%\!=\!2\!\left({}\!^D\!\Gamma_{11}^1{}\!^D\!\Gamma_{12}^2\!-({}\!^D\!\Gamma_{12}^2)^2
%\!+{}\!^D\!\Gamma_{11}^2({}\!^D\!\Gamma_{22}^2\!-{}\!^D\!\Gamma_{12}^1)
%\!+\partial_{x^2}{}\!^D\!\Gamma_{11}^2
%\!-\partial_{x^1}{}\!^D\!\Gamma_{12}^2 \!\right)\!,\\
%\noalign{\medskip}
%\rho(\partial_{x^1},\partial_{x^2})
%\!=\!2\rho_D^{sym}(\partial_{x^1},\partial_{x^2})\\
%\noalign{\medskip}
%\phantom{\rho(\partial_{x^1},\partial_{x^2})}
%\!=\!2\!\left( {}\!^D\!\Gamma_{12}^1{}\!^D\!\Gamma_{12}^2
%\!-{}\!^D\!\Gamma_{11}^2{}\!^D\!\Gamma_{22}^1\right)
%\!-\partial_{x^2}({}\!^D\!\Gamma_{11}^1\!-{}\!^D\!\Gamma_{12}^2)
%\!+\partial_{x^1}({}\!^D\!\Gamma_{12}^1\!-{}\!^D\!\Gamma_{22}^2) ,\\
%\noalign{\medskip}
%\rho(\partial_{x^2},\partial_{x^2})
%\!=\!2\rho_D^{sym}(\partial_{x^2},\partial_{x^2})\\
%\noalign{\medskip}
%\phantom{\rho(\partial_{x^2},\partial_{x^2})}
%\!=\!2\!\left( {}\!^D\!\Gamma_{11}^1{}\!^D\!\Gamma_{22}^1
%\!-{}\!^D\!\Gamma_{12}^2{}\!^D\!\Gamma_{22}^1
%\!+{}\!^D\!\Gamma_{12}^1({}\!^D\!\Gamma_{22}^2\!-{}\!^D\!\Gamma_{12}^1)
%\!-\partial_{x^2}{}\!^D\!\Gamma_{12}^1\!+\partial_{x^1}{}\!^D\!\Gamma_{22}^1 \! \right)\!,
%\end{array}
%\end{equation}
\begin{equation}\label{eq:5-2}
\begin{array}{rcl}
\rho(\partial_{x^1},\partial_{x^1})
&=&2\rho_D^{sym}(\partial_{x^1},\partial_{x^1})\\
\noalign{\medskip}
&=&2\!\left\{{}\!^D\!\Gamma_{11}^1{}\!^D\!\Gamma_{12}^2\!-({}\!^D\!\Gamma_{12}^2)^2
\!+{}\!^D\!\Gamma_{11}^2({}\!^D\!\Gamma_{22}^2\!-{}\!^D\!\Gamma_{12}^1)
\!+\partial_{x^2}{}\!^D\!\Gamma_{11}^2
\!-\partial_{x^1}{}\!^D\!\Gamma_{12}^2 \!\right\}\!,\\
\noalign{\bigskip}
\rho(\partial_{x^1},\partial_{x^2})
&=&2\rho_D^{sym}(\partial_{x^1},\partial_{x^2})\\
\noalign{\medskip}
&=&2\!\left\{ {}\!^D\!\Gamma_{12}^1{}\!^D\!\Gamma_{12}^2
\!-{}\!^D\!\Gamma_{11}^2{}\!^D\!\Gamma_{22}^1\right\}
\!-\partial_{x^2}({}\!^D\!\Gamma_{11}^1\!-{}\!^D\!\Gamma_{12}^2)
\!+\partial_{x^1}({}\!^D\!\Gamma_{12}^1\!-{}\!^D\!\Gamma_{22}^2) ,\\
\noalign{\bigskip}
\rho(\partial_{x^2},\partial_{x^2})
&=&2\rho_D^{sym}(\partial_{x^2},\partial_{x^2})\\
\noalign{\medskip}
&=&2\!\left\{ ({}\!^D\!\Gamma_{11}^1{}\!^D
\!-{}\!^D\!\Gamma_{12}^2){}\!^D\!\Gamma_{22}^1
\!+{}\!^D\!\Gamma_{12}^1{}\!^D\!\Gamma_{22}^2\!-({}\!^D\!\Gamma_{12}^1)^2
\!-\partial_{x^2}{}\!^D\!\Gamma_{12}^1\!+\partial_{x^1}{}\!^D\!\Gamma_{22}^1 \! \right\}\!.
\end{array}
\end{equation}

The following lemma shows that the potential function of a soliton in $T^*\Sigma$ is given in terms of a function and a vector field on $\Sigma$.

\begin{lemma}\label{lemma:form-potential-funtion}
Let $(T^*\Sigma, g_{D,\Phi},f)$ be a gradient Ricci soliton. Then $f=\iota(x)+\pi^*h$ for some vector field $x\in\mathfrak{x}(\Sigma)$ and some function $h:\Sigma\to \mathbb{R}$.
\end{lemma}
\begin{proof}
Let $f:T^*\Sigma\rightarrow\mathbb{R}$ be a function. A straightforward calculation shows that
$\displaystyle\operatorname{Hes}_f(\partial_{x_{i'}},\partial_{x_{j'}})
=\partial^2_{x_{i'}x_{j'}} f(x^1,x^2, x_{1'},x_{2'})$.
On the other hand, since $f$ is the potential function of a gradient Ricci soliton
(i.e., $\operatorname{Hes}_f+\rho=\lambda g_{D,\Phi}$), it follows immediately from the expressions of the metric and the Ricci tensor in \eqref{eq:5-1} and \eqref{eq:5-2} that
$\operatorname{Hes}_f(\partial_{x_{i'}},\partial_{x_{j'}})=0$ for all $i,j=1,2$. Hence the function $f$ expresses as
\begin{equation}
\label{eq:5-4}
f(x^1,x^2, x_{1'},x_{2'})= x_{1'}\pi^*X^1(x^1,x^2)+x_{2'}\pi^*X^2(x^1,x^2)+\pi^* h(x^1,x^2),
\end{equation}
for some smooth functions $X^1$, $X^2$ and $h$ on $\Sigma$.
Now, considering the vector field $x\in\mathfrak{X}(\Sigma)$ defined by
$x=X^i\partial_{x^i}$, the function $f$ is given by $f=\iota(x)+\pi^*h$.
\end{proof}

Assuming the vector field $x\neq 0$, one can specialize coordinates $(x^1,x^2)$ on $\Sigma$ so that $x=\partial_{x^1}$.
Hence, as a consequence of Lemma~\ref{lemma:form-potential-funtion}, the potential function $f$ takes the expression $f=x_{1'}+\pi^*h$ for some function $h$ on $\Sigma$.
Now, considering the following components of the Hessian of a function of the form $f=x_{1'}+\pi^*h$,
%\begin{equation}
%\label{eq.5-3}
$$
\begin{array}{ll}
\displaystyle\operatorname{Hes}_f(\partial_{x^{1}},\partial_{x_{1'}})
={}^D\Gamma_{11}^1
&
\displaystyle\operatorname{Hes}_f(\partial_{x^{1}},\partial_{x_{2'}})
={}^D\Gamma_{11}^2\,,
\\
\noalign{\medskip}
\displaystyle\operatorname{Hes}_f(\partial_{x^{2}},\partial_{x_{1'}})
={}^D\Gamma_{12}^1\,,
&
\displaystyle\operatorname{Hes}_f(\partial_{x^{2}},\partial_{x_{2'}})
={}^D\Gamma_{12}^2\,,
\end{array}
$$
%\end{equation}
and the Ricci components given in \eqref{eq:5-2}, the Ricci soliton equation shows that
\begin{equation}\label{eq:5-5}
{}^D\Gamma_{11}^1={}^D\Gamma_{12}^2=\lambda,
\qquad
{}^D\Gamma_{11}^2={}^D\Gamma_{12}^1=0\,.
\end{equation}
Hence the only possibly non-zero components of the Ricci tensor \eqref{eq:5-2} are
\begin{equation}\label{eq:5-2b}
\rho(\partial_{x^1},\partial_{x^2})
= -\partial_{x^1}{}^D\Gamma_{22}^2 ,
\quad
\rho(\partial_{x^2},\partial_{x^2})
=2\,\partial_{x^1}{}^D\Gamma_{22}^1\,.
\end{equation}

\begin{lemma}\label{lemma:riemannian-extension-steady}
Any non-trivial gradient Ricci soliton $(T^*\Sigma, g_{D,\Phi},f)$ is steady.
\end{lemma}
\begin{proof}
Observe that $(T^*\Sigma,g_{D,\Phi})$ is self-dual for any affine connection $D$ and any symmetric $(0,2)$-tensor field $\Phi$. We adopt the notation of Lemma~\ref{lemma:form-potential-funtion} and assume first that $x$ is a nonzero vector field; hence we work with a potential function $f$ of the form $f=x_{1'}+\pi^*h$. Analyzing the causal character of $\nabla f$, we see from the expression of the Christoffel symbols in \eqref{eq:5-5} that $\|\nabla f\|^2=2\lambda x_{1'}+2\partial_{x^1} h(x^1,x^2)-\Phi_{11}(x^1,x^2)$. Hence, if $\lambda\neq 0$ we have that $\|\nabla f\|^2\neq 0$ on an open dense subset of $T^*\Sigma$
and thus Theorem~\ref{th:1} shows that $(T^*\Sigma,g_{D,\Phi})$ is indeed locally conformally flat. Now, a straightforward calculation of the Weyl tensor shows that
$W(\partial_{x^1},\partial_{x_{1'}},\partial_{x^2},\partial_{x^1})=\frac{1}{2}\partial_{ x^1}{}^D\Gamma_{22}^2$, from where it follows that 
${}^D\Gamma_{22}^2(x^1,x^2)={}^D\Gamma_{22}^2(x^2)$.

Considering again the Ricci soliton equation, we have
$$
\begin{array}{l}
0=
\operatorname{Hes}_f(\partial_{x^2},\partial_{x^2})+\rho(\partial_{x^2},\partial_{x^2})
-\lambda g_{D,\Phi}(\partial_{x^2},\partial_{x^2})\\
\noalign{\medskip}
\phantom{0}
=
\partial^2_{x^2 x^2} h
+{}^D\Gamma_{22}^2(\Phi_{12}-\partial_{x^2} h)-\lambda \Phi_{22}
-\partial_{x^2}\Phi_{12}
+{}^D\Gamma_{22}^1(\Phi_{11}-\partial_{x^1} h)
\\
\noalign{\medskip}
\phantom{0=}
+\frac{1}{2}\partial_{x^1}\Phi_{22}+(2-x_{1'})\partial_{x^1} {}^D\Gamma_{22}^1
\end{array}
$$
and taking the derivative with respect to $x_{1'}$, one has 
$\partial_{x^1} {}^D\Gamma_{22}^1=0$, Now, it follows from \eqref{eq:5-2b} that $(T^*\Sigma,g_{D,\phi})$ is Ricci-flat. This shows that any gradient Ricci soliton on $(T^*\Sigma,g_{D,\phi})$ with potential function $f=\iota x+\pi^* h$ is trivial if $x\neq 0$.
Assume now that the potential function is of the form $f=\pi^* h$.
Then a straightforward calculation shows that 
$$
\operatorname{Hes}_f(\partial_{x^1},\partial_{x_{1'}})+\rho(\partial_{x^1},\partial_{x_{1'}})
-\lambda g_{D,\Phi}(\partial_{x^1},\partial_{x_{1'}})=-\lambda=0,
$$
which shows that the corresponding Ricci soliton is necessarily steady.
\end{proof}

In the next lemma we show that the potential function of the soliton is indeed the pullback to $T^*\Sigma$ of a function on $\Sigma$.

\begin{lemma}\label{lemma:pullback-function}
Let $(T^*\Sigma, g_{D,\Phi},f)$ be a gradient Ricci soliton. Then $f=\pi^* h$ for a function $h:\Sigma\to\mathbb{R}$.
\end{lemma}
\begin{proof}
From Lemma~\ref{lemma:form-potential-funtion} we know that $f=\iota(x)+\pi^*h$ for a vector field $x\in\mathfrak{X}(\Sigma)$ and a function $h:\Sigma\to \mathbb{R}$. 

Assume $x\neq 0$, then we can specialize coordinates $(x^1,x^2)$ on $\Sigma$ so that $x=\partial_{x^1}$ as in previous lemma.
Hence $f=x_{1'}+\pi^*h$. Since $\lambda=0$ by Lemma~\ref{lemma:riemannian-extension-steady}, it follows from equations \eqref{eq:5-5} that ${}^D\Gamma_{11}^1={}^D\Gamma_{12}^2=
{}^D\Gamma_{11}^2={}^D\Gamma_{12}^1=0$.  
Considering now the steady Ricci soliton equation $\operatorname{Hes}_f+\rho=0$, it follows from the component
$\operatorname{Hes}_f(\partial_{x^2},\partial_{x^2})+\rho(\partial_{x^2},\partial_{x^2})=0$
that $\partial_{x^1}{}^D\Gamma_{22}^1=\partial_{x^1}{}^D\Gamma_{22}^2=0$, and thus \eqref{eq:5-2b} shows that $g_{D,\Phi}$ is Ricci flat. Therefore, the assumption $x\neq 0$ forces the soliton to be trivial and we conclude that the potential function reduces to $f=\pi^* h$.
\end{proof}

\begin{proof}[Proof of Theorem\ref{th:1}-(2).]
Let $(T^*\Sigma, g_{D,\Phi},f)$ be a gradient Ricci soliton.
By Lemma~\ref{lemma:pullback-function} the potential function $f$ reduces to $f=\pi^* h$ for some function $h\in\mathcal{C}^\infty(\Sigma)$. A straightforward calculation now shows that the possibly non vanishing terms of $\operatorname{Hes}_f$ are given by
\begin{equation}\label{eq:5-6}
\begin{array}{rcl}
\operatorname{Hes}_f(\partial_{x^1},\partial_{x^1})&=&\operatorname{Hes}^D_{h}(\partial_{x^1},\partial_{x^1}),\\
\noalign{\medskip} \operatorname{Hes}_f(\partial_{x^1},\partial_{x^2})&=&\operatorname{Hes}^D_{h}(\partial_{x^1},\partial_{x^2}),\\ 
\noalign{\medskip}
\operatorname{Hes}_f(\partial_{x^2},\partial_{x^2})&=&\operatorname{Hes}^D_{h}(\partial_{x^2},\partial_{x^2}).
\end{array}
\end{equation}
Since the soliton is steady by Lemma~\ref{lemma:riemannian-extension-steady}, expressions in Equations \eqref{eq:5-2} and \eqref{eq:5-6} show that the equation $\operatorname{Hes}_f+\rho=0$ is equivalent to $\operatorname{Hes}^D_h+2\rho_D^{sym}=0$, which finishes the proof.
\end{proof}

\begin{remark}
\rm
Observe that the existence of a gradient Ricci soliton on $(T^*\Sigma, g_{D,\Phi})$ is independent of the choice of the symmetric $(0,2)$-tensor field $\Phi$. Hence, any affine gradient Ricci soliton $(\Sigma, D,h)$ induces an infinite family of steady gradient Ricci solitons $(T^*\Sigma, g_{D,\Phi},\pi^*h)$. Moreover, all these gradient Ricci solitons are isotropic since $\|\nabla \pi^*h\|^2=0$.
\end{remark}

\begin{remark}\label{re:14}
\rm
The conformal structure of  $(T^*\Sigma,g_{D,\Phi})$ is related to the projective structure of the affine surface $(\Sigma,D)$. 
Indeed, it follows by a long but straightforward calculation of the Weyl tensor of $g_{D,\Phi}$ that the possibly non-zero components, up to curvature symmetries are
$$
\begin{array}{rcl}
W(\partial_{x^1},\partial_{x^2},\partial_{x^1},\partial_{x_{1'}})
&\!\!\!=\!\!\!&W(\partial_{x^1},\partial_{x^2},\partial_{x^2},\partial_{x_{2'}})=\frac{1}{2}\{\rho^D(\partial_{x^1},\partial_{x^2})-\rho^D(\partial_{x^2},\partial_{x^1})\}\\
\noalign{\medskip}
W(\partial_{x^1},\partial_{x^2},\partial_{x^1},\partial_{x^2})
&\!\!\!=\!\!\!& \Theta (\Phi)+x_{1'}\left\{(D_{\partial_{x^1}}\rho^D)(\partial_{x^2},\partial_{x^2})
-(D_{\partial_{x^2}}\rho^D)(\partial_{x^2},\partial_{x^1}) \right\}
\\
\noalign{\medskip}
&&\phantom{\Theta (\Phi)}+x_{2'}\left\{(D_{\partial_{x^2}}\rho^D)(\partial_{x^1},\partial_{x^1})
-(D_{\partial_{x^1}}\rho^D)(\partial_{x^1},\partial_{x^2}) \right\},
\end{array}
$$
where $\Theta$ is a polynomial on the derivatives of the components $\Phi_{ij}$ up to order two. It follows that if $(T^*\Sigma,g_{D,\Phi})$ is locally conformally flat, then the Ricci tensor $\rho^D$ of the affine connection $D$ and its covariant derivative $D\rho^D$ are symmetric, i.e., the affine connection $D$ is projectively flat with symmetric Ricci tensor. 
\end{remark}

\emph{Proof of Theorem~\ref{th:4D-lcf}.}
Locally conformally flat gradient Ricci solitons are locally warped products in the non-isotropic case, which also includes the Riemannian setting (note that the proof in \cite{BGG} remains valid for arbitrary signature) from where Assertion (1) follows.
Moreover, Theorem~\ref{th:4D-lcf}-(2) was proven in \cite{BGG} where it is shown that isotropic locally conformally flat Lorentzian gradient Ricci solitons are plane waves as specified. 
In order to proof Assertion (3), recall from Theorem \ref{th:1}
that any isotropic locally conformally flat gradient Ricci soliton is locally a deformed Riemannian extension in signature $(2,2)$. Moreover, it follows from the work of Afifi \cite{afifi} that any deformed Riemannian extension $g_{D,\Phi}$ which is locally conformally flat is indeed a Riemannian extension $g_D$ after a suitable change of coordinates on the affine surface $(\Sigma, D)$. Hence, such structure is a locally conformally flat gradient Ricci soliton if and only if it is the Riemannian extension of a projectively flat affine gradient Ricci soliton. \qed 

\section{Affine gradient Ricci solitons}\label{se-6}

In this section we analyze the existence of affine gradient Ricci solitons on affine surfaces, i.e.,  $(\Sigma, D, h)$ satisfying 
$
\operatorname{Hes}_h^D+2\rho_{sym}^D=0$
for some function $h:\Sigma\rightarrow\mathbb{R}$. As the results in the previous section showed,
$(\Sigma, D, h)$ is an affine gradient Ricci soliton if and only if $(T^*\Sigma, g_{D},\pi^*h)$ is a gradient Ricci soliton. It then follows from Lemma \ref{lemma:grs-formulas}-(3)
that
$R(x,y,\nabla f,z)=(\nabla_x \rho)(y,z)-(\nabla_y \rho)(x,z)$, and thus the curvature operator $R(x,y)$ satisfies
\begin{equation}
\label{eq:compatibility}
R(x,y)\nabla(\pi^* h)=(\nabla_x \operatorname{Ric})y-(\nabla_y \operatorname{Ric})x,
\end{equation}
for all vector fields $x$, $y$ on $T^*\Sigma$,
where $\operatorname{Ric}$ is the Ricci operator on  $(T^*\Sigma, g_{D})$. 
Furthermore, if at some point $\xi=(p,\omega)\in T^*\Sigma$ the vector fields $x$ and $y$ are given by
$x_\xi=\alpha_i\partial_{x^i}+\alpha_{i'}\partial_{x_{i'}}$ and
$y_\xi=\beta_i\partial_{x^i}+\beta_{i'}\partial_{x_{i'}}$, then the matrix of the
curvature operator $R(x,y)_\xi$ with respect
to the coordinate basis $\{\partial_{x^i},\partial_{x_{i'}}\}$ is of the form
\begin{equation}\label{eq: matriz IP extension de Riemann}
R(x,y)_\xi = \left(
\begin{array}{cc}
     R^D(\pi_*x,\pi_*y)_p & {0}
    \\
    {\ast} & - R^D(\pi_*x,\pi_*y)_p^t
\end{array}
\right),
\end{equation}
where $R^D(\pi_*x,\pi_*y)_p$ is the matrix of the curvature operator corresponding to
$\pi_*x=\alpha_i\partial_i$ and $\pi_*y=\beta_i\partial_i$ (see \cite{CLGRVL} for details). Moreover the curvature tensor is given by the Ricci as $ R(x,y)z=\rho^D(y,z)x-\rho^D(x,z)y$ for arbitrary vector fields $x,y,z$.
Now, since the gradient of the potential function $\pi^*h$ is given by
$\nabla(\pi^*h)=\partial_{x^1}h\,\partial_{x_{1'}}+\partial_{x^2}h\,\partial_{x_{2'}}$ it follows that 
$$
\begin{array}{rcl}
R(x,y)\nabla(\pi^*h)&=&\{\rho^D(\pi_*y,\pi_*x)\partial_{x^1}h+\rho^D(\pi_*y,\pi_*y)\partial_{x^2}h\}\partial_{x_{1'}}\\
\noalign{\medskip}
&&-\{\rho^D(\pi_*x,\pi_*x)\partial_{x^1}h+\rho^D(\pi_*x,\pi_*y)\partial_{x^2}h\}\partial_{x_{2'}}\,.
\end{array}
$$

This makes of special significance the cases when the curvature operator is invertible and when the affine surface $(\Sigma,D)$ is projectively flat. In the first case the potential function is characterized by the equation 
$\nabla(\pi^* h)=R(x,y)^{-1}[(\nabla_x \operatorname{Ric})y-(\nabla_y \operatorname{Ric})x]$,
while in the second case, it must satisfy $R(x,y)\nabla(\pi^* h)=0$. In particular no non-trivial affine gradient Ricci soliton may exist on a projectively flat affine surface with non-degenerate Ricci tensor.

%In this section we examine the existence of affine gradient Ricci solitons on locally symmetric and locally homogeneous affine surfaces.

\subsection{Projectively flat locally homogeneous affine gradient Ricci solitons}
Locally homogeneous affine surfaces were described by Opozda in \cite{Opozda}, who showed that if an affine surface $(\Sigma,D)$ is locally homogeneous
then either $D$ is the Levi-Civita connection of a surface with constant Gaussian curvature or,
otherwise, there exist local coordinates $(x^1,x^2)$  such that
$D$ has Christoffel symbols given by:
\begin{equation} \label{eq: Type-A}
\mbox{Type A:\quad ${}^D\Gamma_{ij}^k(x^1,x^2)=\gamma_{ij}^k$, or}
%\begin{array}{ll}
%D_{\partial_{x^1}}\partial_{x^1} = a \partial_{x^1} + b
%    \partial_{x^2}, &
%D_{\partial_{x^1}}\partial_{x^2} = c \partial_{x^1} + d
%    \partial_{x^2}, \\
%\noalign{\medskip}
%    D_{\partial_{x^2}}\partial_{x^2} = e \partial_{x^1} + f \partial_{x^2},&
%\end{array}
    \end{equation}
\begin{equation}\label{eq: Type-B}
   \mbox{Type B:\quad ${}^D\Gamma_{ij}^k(x^1,x^2)=\frac{1}{x^1}\gamma_{ij}^k$, }    
%       D_{\partial_{x^1}}\partial_{x^1}=\frac{a}{x^1}\partial_{x^1}\!+\!\frac{b}{x^1}
%    \partial_{x^2},\quad
%    D_{\partial_{x^1}}\partial_{x^2}=\frac{c}{x^1}\partial_{x^1}\!+\!\frac{d}{x^1}
%    \partial_{x^2},\quad
%   D_{\partial_{x^2}}\partial_{x^2}=\frac{e}{x^1}\partial_{x^1}+\!\frac{f}{x^1}\partial_{x^2}.
\end{equation}
for some constants $\gamma_{ij}^k$.

Recall that a connection $D$ on $\Sigma$ is locally homogeneous if and
only if, in a neighborhood of each point, there exist two linearly
independent affine Killing vector fields (i.e., vector fields whose local flows leave the connection invariant $\mathcal{L}_xD=0$).
If a connection is of Type A, then it has constant Christoffel
symbols in some local coordinates $(x^1,x^2)$ and hence $\{\partial_{x^1},\partial_{x^2}\}$ are linearly independent commuting affine Killing vector fields. Further, if a connection is of Type B, all Christoffel symbols are given by \eqref{eq: Type-B} in some local coordinates $(x^1,x^2)$
and one has that $\{x=\partial_{x^2},y=x^1\partial_{x^1}+x^2\partial_{x^2}\}$ are linearly independent affine Killing vector fields satisfying $[x,y]=x$. 
The converse statement holds true, and types A and B above are characterized by the existence of linearly independent affine Killing vector fields $x$, $y$ satisfying $[x,y]=0$ and $[x,y]=x$, respectively \cite{AM-K}.
It is worth emphasizing here that types A and B are not disjoint. Indeed, their intersection was studied in \cite{CLGRVL}, where it is shown that any Type B connection which is projectively flat with symmetric, degenerate and recurrent Ricci tensor is also of Type A.

Any locally homogeneous affine connection of Type A has symmetric Ricci tensor and is projectively flat 
(cf. \cite{CLGRVL}). Hence, the compatibility condition \eqref{eq:compatibility} shows that affine gradient Ricci solitons with curvature operator (equivalenly Ricci tensor) of rank two are trivial. We assume the curvature operator has rank-one from now on. 
Since the Christoffel symbols are constant, one may use a linear transformation to ensure that the only non-zero component of the Ricci tensor is $\rho^D(\partial_{x^1},\partial_{x^1})$, which is equivalent to the equations
$$
\begin{array}{l}
\rho^D(\partial_{x^1},\partial_{x^2})=\gamma_{12}^1\gamma_{12}^2-\gamma_{11}^2\gamma_{22}^1=0,
\\
\noalign{\medskip}
\rho^D(\partial_{x^2},\partial_{x^2})=\gamma_{22}^1(\gamma_{11}^1-\gamma_{12}^2) +\gamma_{12}^1(\gamma_{22}^2-\gamma_{12}^1)=0.
\end{array}
$$

A straightforward calculation shows that the non-flat solutions are given by setting  $\gamma_{12}^1=0$ and  $\gamma_{22}^1=0$ at \eqref{eq: Type-A}. Moreover, in this case it follows from \eqref{eq:compatibility} that the potential function of the soliton is of the form $h(x^1)$, where it is a solution of the linear equation
$$
h''(x^1)- \gamma_{11}^1 h'(x^1)= -2 \rho^D(\partial_{x^1},\partial_{x^1})
= -2(\gamma_{11}^1\gamma_{12}^2+\gamma_{11}^2\gamma_{22}^2-(\gamma_{12}^2)^2).
$$
This shows that a homogeneous affine connection of Type A is a non-trivial affine gradient Ricci soliton if and only if the Ricci tensor is of rank one. 

\medskip

Next we consider Type B locally homogeneous affine connections given by \eqref{eq: Type-B}.
Let $(\Sigma, D)$ be such an affine surface. A simple observation of the Weyl tensor of the Riemannian extension $(T^*\Sigma, g_D)$ as in Remark \ref{re:14} shows that $(\Sigma, D)$ is projectively flat if and only if the Ricci tensor is symmetric (i.e., $\gamma_{12}^1=-\gamma_{22}^2$ in \eqref{eq: Type-B})
and the covariant derivative of the Ricci tensor satisfies $(D_x\rho^D)(y,z)=(D_y\rho^D)(x,z)$, 
i.e., the Christoffel symbols satisfy the equations
\begin{equation}
\label{eq:typeB-pll-1}
2 \gamma_{12}^1\gamma_{12}^2-3\gamma_{11}^2\gamma_{22}^1+(2+\gamma_{11}^1)\gamma_{22}^2=0, \; 
(\gamma_{12}^1)^2-2\gamma_{12}^1\gamma_{22}^2+\gamma_{22}^1(1-\gamma_{11}^1+2\gamma_{12}^2)=0\,.
\end{equation}

Then a Type B locally homogeneous affine connection is projectively flat if and only if $\gamma_{12}^1=-\gamma_{22}^2$ and one of the following holds
\begin{enumerate}
\item[(i)] $\gamma_{12}^1=0$, $\gamma_{22}^1=0$, or
\smallskip
\item[(ii)] $\gamma_{12}^2=\frac{1}{2}(\gamma_{11}^1-1-3\frac{(\gamma_{12}^1)^2}{\gamma_{22}^1})$, $\gamma_{11}^2=-\frac{(\gamma_{12}^1)^3}{(\gamma_{22}^1)^2}-\frac{\gamma_{12}^1}{\gamma_{22}^1}$, $\gamma_{22}^1\neq 0$\,.
\end{enumerate}

The Ricci tensor of a connection given by $(i)$ above is of rank one and, moreover, it follows after some calculations that it is recurrent  ($D\rho=\omega\oplus \rho$ with $\omega=-\frac{2}{x^1}(1+\Gamma_{11}^1)dx^1$). Hence, any Type B connection given by $(i)$ is also of Type A, and thus an affine gradient Ricci soliton.

Let now $D$ be a Type B connection given by $(ii)$. Since $D$ is projectively flat, non-trival affine gradient Ricci solitons may occur only if the Ricci tensor is degenerated, which in terms of the Christoffel symbols is equivalent to
$$
\left( (\gamma_{12}^1)^2+\gamma_{22}^1-\gamma_{11}^1\gamma_{22}^1\right)
\left( (\gamma_{12}^1)^2-(3+\gamma_{11}^1)\gamma_{22}^1\right)=0.
$$
The condition $(\gamma_{12}^1)^2+\gamma_{22}^1-\gamma_{11}^1\gamma_{22}^1=0$, implies that the Ricci tensor vanishes. Therefore we assume that 
$(\gamma_{12}^1)^2-(3+\gamma_{11}^1)\gamma_{22}^1=0$. Computing again the Ricci tensor, one has
$$
\rho^D=-\frac{2}{(x^1)^2}
\left(\begin{array}{cc}
\frac{(\gamma_{12}^1)^2}{\gamma_{22}^1}&\gamma_{12}^1\\
\noalign{\medskip}
\gamma_{12}^1&\gamma_{22}^1
\end{array}\right).
$$
Let $h(x^1,x^2)$ be a smooth function on $\Sigma$. Considering the compatibility condition \eqref{eq:compatibility} one has that $\gamma_{22}^1\partial_{x^1}h-\gamma_{12}^1\partial_{x^2}h=0$, and thus 
 $\partial_{x^1}h=\frac{\gamma_{12}^1}{\gamma_{22}^1}\partial_{x^2}h$.
 Now the different components on the Ricci soliton equation 
$\operatorname{Hes}_h^D+2\rho_{sym}^D=0$ become
$$
\begin{array}{rcl}
\operatorname{Hes}_h^D(\partial_{x^2},\partial_{x^2})+2\rho_{sym}^D(\partial_{x^2},\partial_{x^2})
&=&\partial^2_{x^2 x^2}h-\frac{4}{(x^1)^2}\gamma_{22}^1\,,
\\
\noalign{\medskip}
\operatorname{Hes}_h^D(\partial_{x^1},\partial_{x^2})+2\rho_{sym}^D(\partial_{x^1},\partial_{x^2})
&=&\partial^2_{x^1 x^2}h+\frac{2}{x^1}\partial_{x^2}h-\frac{4}{(x^1)^2}\gamma_{12}^1\,,
\\
\noalign{\medskip}
\operatorname{Hes}_h^D(\partial_{x^1},\partial_{x^1})+2\rho_{sym}^D(\partial_{x^1},\partial_{x^1})
&=&\partial^2_{x^1 x^1}h+\frac{4}{x^1}\partial_{x^1}h
-\frac{4(\gamma_{12}^1)^2}{(x^1)^2\gamma_{22}^1}\,.
\end{array}
$$
The first equation above shows that $\partial^2_{x^2x^2}h=\frac{4}{(x^1)^2}\gamma_{22}^1$, and thus $\partial^2_{x^1x^2}h=\frac{4}{(x^1)^2}\gamma_{12}^1$. Then the second equation above shows that  $\partial_{x^2}h=0$, which implies that $h$ is indeed a constant. This shows that Type B homogeneous connections given by $(ii)$ result in trivial affine gradient Ricci solitons.

Summarizing the above, we have

\begin{theorem}
\label{th:homogeneous-agrs}
Let $(\Sigma,D)$ be a projectively flat locally homogeneous affine surface. Then it is a non-trivial affine gradient Ricci soliton if and only if the Christoffel symbols are constant in suitable coordinates and the Ricci tensor is of rank one.
\end{theorem}

\begin{remark}\label{re:homogeneous-agrs-1}
\rm
As a consequence of the previous analysis one has that the kernel of the Ricci tensor of any projectively flat homogeneous affine gradient Ricci soliton is one-dimensional (there exist coordinates $(x^1,x^2)$ where $\operatorname{ker}\rho^D=\operatorname{span}\{\partial_{x^2}\}$) and moreover it is parallel
(since ${}^D\Gamma_{12}^1(x^1,x^2)=0$ and ${}^D\Gamma_{22}^1(x^1,x^2)=0$).
\end{remark}

\begin{remark}\rm \label{re:homogeneous-agrs-2}
An affine surface $(\Sigma,D)$ is locally symmetric if the connection satisfies $D\rho^D=0$. Any locally symmetric connection is projectively flat and has symmetric Ricci tensor \cite{Op1}. Hence it follows from Theorem \ref{th:homogeneous-agrs} that a locally symmetric affine surface is a non-trivial affine gradient Ricci soliton if and only if the Ricci tensor has rank one.
Any such surface can be described, in adapted coordinates $(x^1,x^2)$ where the Christoffel symbols are constant as in \eqref{eq: Type-A}, by the constrains $\gamma_{11}^1=0$, $\gamma_{12}^1=0$, $\gamma_{22}^1=0$, and the potential function of the soliton is given by the equation $h''(x^1)= -2 \rho^D(\partial_{x^1},\partial_{x^1})$. 

Note that while Riemannian extensions $g_D$ of a locally symmetric affine connection $D$ are locally symmetric, deformed Riemannian extensions $g_{D,\Phi}$ are not symmetric for a general $\Phi$.
\end{remark}

\begin{remark}
\rm
Although the examples of locally homogeneous affine gradient Ricci solitons in Theorem~\ref{th:homogeneous-agrs} are projectively flat by hypothesis, one can build examples which are not projectively flat from manifolds of Type B above. Let $(\Sigma,D)$ be an affine surface given by \eqref{eq: Type-B} and specify the Christoffel symbols so that the only non-zero ones are given by  ${}^D\Gamma_{11}^1(x^1,x^2)=-\frac{1}{x^1}$ and ${}^D\Gamma_{22}^1(x^1,x^2)=\frac{\gamma_{22}^1}{x^1}$. A straightforward calculation shows that it is an affine gradient Ricci soliton with potential function 
$h(x^1,x^2)=-4\ln(x^1)+\alpha x^2+\beta$. Observe that $(\Sigma,D)$ has Ricci tensor symmetric of rank one \cite{CLGRVL} and recurrent with $\omega=-\frac{2}{x^1}dx^1$.
\end{remark}

\subsection{Non-projectively flat affine gradient Ricci solitons}

Through this section we will examine a special family of affine surfaces $(\Sigma,D)$ which generalizes that of locally symmetric affine connections and results in new examples of non-projectively flat affine gradient Ricci solitons. We consider affine connections whose Ricci tensor is symmetric of rank one and such that the kernel of $\rho^D$ is  parallel. These connections were described by Opozda \cite{Op1}, who showed the existence of adapted coordinates $(x^1,x^2)$ where the only nonzero Christoffel symbols are
\begin{equation}\label{eq:proj-flat-1}
{}^D\Gamma_{12}^1 \quad \mbox{and} \quad {}^D\Gamma_{22}^1\,\quad \mbox{where} \quad \partial_{x^1}{}^D\Gamma_{12}^1=0 \,.
\end{equation}
Moreover, the connection $D$ is projectively flat if and only if 
$\partial^2_{x^1 x^1}{}^D\Gamma_{22}^1=0$ and it is locally symmetric if and only if 
$\partial^2_{x^1 x^1}{}^D\Gamma_{22}^1=0$ and 
$\partial^2_{x^1 x^2}{}^D\Gamma_{22}^1=\partial^2_{x^2 x^2}{}^D\Gamma_{12}^1
+2\,\Gamma_{12}^1\partial_{x^2}{}^D\Gamma_{12}^1$.

Let $h(x^1,x^2)$ be an arbitrary function on $\Sigma$. Then the Ricci soliton equation $\operatorname{Hes}^D_h+2\rho^D_{sym}=0$ reduces to
\begin{equation}
\label{eq:proj-flat-2}
\begin{array}{rcl}
\displaystyle
\operatorname{Hes}^D_h(\partial_{x^1},\partial_{x^1})+2\rho^D_{sym}(\partial_{x^1},\partial_{x^1})
&=&\displaystyle\partial^2_{x^1 x^1}h\,,\\
\noalign{\medskip}
\displaystyle
\operatorname{Hes}^D_f(\partial_{x^1},\partial_{x^2})+2\rho^D_{sym}(\partial_{x^1},\partial_{x^2})
&=&\displaystyle\partial^2_{x^1 x^2}h-{}^D\Gamma_{12}^1 \partial_{x^1}h\,,\\
\noalign{\medskip}
\displaystyle
\operatorname{Hes}^D_f(\partial_{x^1},\partial_{x^2})+2\rho^D_{sym}(\partial_{x^2},\partial_{x^2})
&=&\displaystyle\partial^2_{x^2 x^2}h-{}^D\Gamma_{22}^1\partial_{x^1}h
+2\partial_{x^1}{}^D\Gamma_{22}^1\\
\noalign{\medskip}
&&\displaystyle\phantom{\partial^2_{x^2 x^2}h}
-2\partial_{x^2}{}^D\Gamma_{12}^1-2({}^D\Gamma_{12}^1)^2
\,.\end{array}
\end{equation}

First of all, observe that if $D$ is projectively flat, then for any function of the form $h(x^2)$, the system \eqref{eq:proj-flat-2} reduces to the equation
$$
h''(x^2)=2\partial_{x^2}{}^D\Gamma_{12}^1(x^2)+2({}^D\Gamma_{12}^1(x^2))^2
-2\partial_{x^1}{}^D\Gamma_{22}^1(x^1,x^2),
$$ 
which is meaningful, in the sense that the equation always has local solution and therefore provides new examples, since $\partial^2_{x^1x^1}{}^D\Gamma_{22}^1(x^1,x^2)=0$ due to projective flatness.

In the general situation, it follows from the first equation in \eqref{eq:proj-flat-2} that the potential function $h$ splits as $h(x^1,x^2)=\tilde{h}(x^2)+x^1\hat{h}(x^2)$. Moreover, the second equation in \eqref{eq:proj-flat-2} now shows that the function $\hat{h}$ is completely determined by the linear homogeneous equation $\hat{h}'(x^2)={}^D\Gamma_{12}^1(x^2)\hat{h}(x^2)$, and thus
$\hat{h}(x^2)=\kappa e^{\Xi_{12}^1(x^2)}$, where $\Xi_{12}^1(x^2)$ is a primitive of ${}^D\Gamma_{12}^1(x^2)$ and $\kappa\in\mathbb{R}$.
Now, considering the third equation above, one has
$$
\begin{array}{rcl}
\tilde{h}''(x^2)&=&2({}^D\Gamma_{12}^1(x^2))^2
+2\partial_{x^2}{}^D\Gamma_{12}^1(x^2)\\
\noalign{\medskip}
&&+\hat{h}(x^2){}^D\Gamma_{22}^1(x^1,x^2)-2\partial_{x^1}{}^D\Gamma_{22}^1(x^1,x^2)-x^1 \hat{h}''(x^2)\,.
\end{array}
$$

The necessary and sufficient condition so that the equation above is meaningful is obtained by taking the derivative with respect to $x^1$: 
\begin{equation}
\label{eq:proj-flat-2b}
\hat{h}''(x^2)-\hat{h}(x^2)\partial_{x^1}{}^D\Gamma_{22}^1(x^1,x^2)+2\partial^2_{x^1x^1}{}^D\Gamma_{22}^1(x^1,x^2)=0.
\end{equation}
 Now, assuming $\hat{h}(x^2)\neq 0$, one has that the previous equation determines the Christoffel symbol ${}^D\Gamma_{22}^1(x^1,x^2)$ as follows
\begin{equation}
\label{eq:proj-flat-3}
\begin{array}{rcl}
{}^D\Gamma_{22}^1(x^1,x^2)&=&
\displaystyle\frac{2\alpha(x^2)}{\kappa} e^{\left(\displaystyle\frac{1}{2}\kappa x^1 e^{\Xi_{12}^1(x^2)}-\Xi_{12}^1(x^2)\right)}+\beta(x^2)\\
\noalign{\medskip}
&&\displaystyle+x^1\left(({}^D\Gamma_{12}^1(x^2))^2+\partial_{x^2}{}^D\Gamma_{12}^1(x^2)\right),
\end{array}
\end{equation}
for arbitrary functions $\alpha(x^2)$, $\beta(x^2)$.
Moreover, in this case the Ricci soliton equation \eqref{eq:proj-flat-1} reduces to 
\begin{equation}
\label{eq:proj-flat-4} 
{\tilde{h}''(x^2)=\kappa\beta(x^2)e^{\Xi_{12}^1(x^2)}=\beta(x^2)\hat h(x^2)}.
\end{equation}

Summing up the above, we have

\begin{theorem}
\label{th:proj-flat-0}
Let $(\Sigma,D)$ be an affine surface whose Ricci tensor is  symmetric of rank one and has parallel kernel. Then it is an affine gradient Ricci soliton if and only if either it is projectively flat or, otherwise there exist coordinates $(x^1,x^2)$ where the non-zero Christoffel symbols are given by \eqref{eq:proj-flat-1}, \eqref{eq:proj-flat-3}. Moreover, in the later case the potential function of the soliton is given by $h(x^1,x^2)=\tilde{h}(x^2)+x^1\hat{h}(x^2)$, where $\hat{h}(x^2)=\kappa e^{\Xi_{12}^1(x^2)}$ and $\tilde{h}$ is given by \eqref{eq:proj-flat-4}.
\end{theorem}

%\begin{remark}\rm
%Note that the affine connections in Theorem \ref{th:proj-flat-0} are affine gradient Ricci solitons, which are not locally symmetric unless they are projectively flat (which corresponds to the case $\alpha(x^2)=0$). Moreover, note that while the Riemannian extension $g_D$ of a locally symmetric affine connection is also locally symmetric, the deformed Riemannian extension $g_{D,\Phi}$ is not symmetric in the generic setting.
%\end{remark}

\begin{remark}
\rm
So far all the examples previously discussed of affine gradient Ricci solitons have Ricci tensor of rank one. However, there also exist non-trivial examples with Ricci tensor of rank two. An immediate application of the compatibility condition \eqref{eq:compatibility} shows that such an example must fail to be projectively flat, as already mentioned. We consider locally homogeneous affine connections of Type B defined by \eqref{eq: Type-B} and specify the Christoffel symbols to be given by
${}^D\Gamma_{11}^1(x^1,x^2)=\frac{1}{x^1}\left(\gamma_{12}^2\pm\sqrt{1+2\gamma_{12}^2}\right)$,
${}^D\Gamma_{12}^2(x^1,x^2)=\frac{\gamma_{12}^2}{x^1}$,
${}^D\Gamma_{22}^1(x^1,x^2)=\frac{\gamma_{22}^1}{x^1}$
for non-zero constants $\gamma_{12}^2$, $\gamma_{22}^1$ such that $\gamma_{12}^2\geq\frac{1}{2}$.
Then $(\Sigma,D)$ is a locally homogeneous affine gradient Ricci soliton with potential function
$h(x^1,x^2)=\kappa \pm 2\left( \mp 1\pm\sqrt{1+2\gamma_{12}^2} \right)
\ln\left( x^1\mp x^1\sqrt{1+2\gamma_{12}^2}\right)$.

Moreover, the Ricci tensor $\rho^D$ defines a metric of constant scalar curvature 
$\tau=\frac{1\mp\sqrt{1+2\gamma_{12}^2}}{(\gamma_{12}^2)^2}$ on $\Sigma$. 
\end{remark}

\end{document}